\documentclass[12pt]{article}
\usepackage{tikz}
\usepackage{ucs}
\usepackage{amssymb}
\usepackage{amsthm}
\usepackage{amsmath}
\usepackage{latexsym}
\usepackage[cp1251]{inputenc}
\usepackage[english]{babel}
\usepackage{graphicx}
\usepackage{wrapfig}
\usepackage{caption}
\usepackage{subcaption}
\usepackage{txfonts}
\usepackage{lmodern}
\usepackage{mathrsfs}
\usepackage{algorithm}
\usepackage{multicol}
\usepackage{enumerate}
\usepackage{titlesec}
\usepackage{paralist}
\usepackage{mathtools}
\usepackage[hidelinks]{hyperref}

\newtheorem{thm}{Theorem}%[section]

\newtheorem{claim}[thm]{Claim}
\newtheorem{conj}[thm]{Conjecture}
\newtheorem{cor}[thm]{Corollary}
\newtheorem{prop}[thm]{Proposition}

{}
\newtheorem{theorem}[thm]{Theorem}

\newtheorem{example}[thm]{Example}
\newtheorem*{example*}{Example}

\newtheorem*{definition*}{Definition}
\newtheorem*{remark*}{Remark}
\newtheorem{question}[thm]{Question}

\titleformat{\paragraph}
{\normalfont\normalsize\bfseries}{\theparagraph}{1em}{}
\titlespacing*{\paragraph}
{0pt}{3.25ex plus 1ex minus .2ex}{1.5ex plus .2ex}

\newcommand*{\myproofname}{Proof}

\usepackage{wasysym}

\usepackage{fullpage}

\def\qed{\hfill\ifhmode\unskip\nobreak\fi\qquad\ifmmode\Box\else\hfill$\Box$\fi}

\title{Packing $(1,1,2,4)$-coloring of subcubic outerplanar graphs}
\date{\today}
\author{
Alexandr Kostochka \thanks{Department of Mathematics, University of Illinois at Urbana--Champaign, IL, USA and
Sobolev Institute of Mathematics, Novosibirsk 630090, Russia, kostochk@math.uiuc.edu. Research of this author is supported in part by NSF grant
 DMS-1600592,  NSF RTG Grant DMS-1937241,  Award RB17164 of the UIUC Campus Research Board, and  grant 19-01-00682 of the Russian Foundation for Basic Research.}
 \and Xujun Liu\thanks{Coordinated Science Laboratory, University of Illinois at Urbana--Champaign, IL, USA, xliu150@illinois.edu; the work was partially done while X. Liu was a PhD student at Department of Mathematics, University of Illinois at Urbana--Champaign. Research of this author was supported by Emerging Frontiers of Science of Information (A2209)1-482031-239010-191100.
}
 }

\begin{document}
	\maketitle

\begin{abstract}
For $1\leq s_1 \le s_2 \le \ldots \le s_k$ and a graph $G$, {a {\em packing $(s_1, s_2, \ldots, s_k)$-coloring} of $G$ is a partition of $V(G)$ into sets $V_1, V_2, \ldots, V_k$ such that, for each $1\leq i \leq k$,} the distance between any two distinct $x,y\in V_i$ is at least $s_i + 1$. The {\em packing chromatic number}, $\chi_p(G)$, of a graph $G$ is the smallest $k$ such that $G$ has a packing $(1,2, \ldots, k)$-coloring. It is known that there are trees of maximum degree 4 and subcubic graphs $G$
with arbitrarily large $\chi_p(G)$. Recently, there was a series of papers on packing $(s_1, s_2, \ldots, s_k)$-colorings of subcubic graphs in various classes.
 We show that every $2$-connected subcubic outerplanar graph has a packing $(1,1,2)$-coloring and every subcubic outerplanar graph is packing $(1,1,2,4)$-colorable.
  Our results are sharp in the sense that there are $2$-connected subcubic outerplanar graphs that are not packing $(1,1,3)$-colorable and there are subcubic outerplanar graphs that are not packing $(1,1,2,5)$-colorable. We also show subcubic outerplanar graphs that are not packing $(1,2,2,4)$-colorable and not packing $(1,1,3,4)$-colorable.
\end{abstract}

\section{Introduction}	

For a non-decreasing sequence $S=(s_1,s_2,...,s_k)$ of positive integers, a {\em packing $S$-coloring} of a graph $G$ is a partition of $V(G)$ into sets $V_1,...,V_k$ such that olor{blue}{, for each $1 \le i \le k$,} the distance between any two distinct $x,y \in V_i$ is at least $s_i+1$. A {\em packing $k$-coloring} is  a packing $(1,2,...,k)$-coloring. The {\em packing chromatic number}, $\chi_p(G)$ (we will use the abbreviation PCN for short), of a graph $G$ is the minimum $k$ such that $G$ has a packing $k$-coloring.

Packing $k$-coloring was introduced in 2008 by Goddard,  Hedetniemi, Hedetniemi,  Harris and  Rall~\cite{GHHHR}
%(under the name {\em broadcast coloring})
motivated by frequency assignment problems in broadcast networks.
There are more than 50 papers on the topic
(e.g.~\cite{ANT1,BF, BFKR, BKR1,BKR2,BKRW1,BKRW2,BKRW3,CJ1, FG1,FKL1,G1,GT1, GHT, KV1, LBS2, S1, TV1} and references in them). In particular, Fiala and Golovach~\cite{FG1} proved that finding the PCN of a graph is NP-complete even in the class of trees. Sloper~\cite{S1} showed that the infinite complete ternary tree (every vertex has $3$ child vertices) has unbounded PCN.

The question  whether PCN is bounded in the class of subcubic graphs
was discussed in several papers (e.g., in~\cite{BKRW1,BKRW2,GT1,LBS1,S1}) and answered in the negative 
in~\cite{BKL}. Bre\v sar and Ferme~\cite{BF} later provided an explicit family of subcubic graphs with unbounded PCN. This stimulated studying subclasses of subcubic graphs with bounded PCN.

One of the studied classes was the class of subdivisions of subcubic graphs. Recall that
the {\em subdivision,} $D(G)$, of a graph $G$ is the graph obtained from $G$ by replacing each edge by a path of length two. 
In particular, Gastineau and Togni~\cite{GT1} asked whether $\chi_p(D(G))\leq 5$ for every subcubic graph $G$ and Bre\v sar, Klav\v zar, Rall, and Wash~\cite{BKRW2} later conjectured this.
%(i.e., a graph with maximum degree at most $3$)

\begin{conj}[Bre\v sar, Klav\v zar, Rall, and Wash~\cite{BKRW2}]\label{conjbkrw}
Let $G$ be a subcubic graph. Then $\chi_p(D(G))\leq 5$.
\end{conj}

In 2018~\cite{BKL2} it was shown that the PCN of the subdivision of every subcubic graph is
at most $8$. Gastineau and Togni~\cite{GT1} pointed out at the following connection between the bounds on 
packing $S$-colorings of a graph $G$ and the packing $S$-colorings of $D(G)$:

\begin{prop}[Gastineau and Togni~\cite{GT1}, Proposition 1]\label{extension}
Let $G$ be a graph and $S
=(s_1,...,s_k)$ be a non-decreasing sequence of integers. If $G$ is packing $S$-colorable, then $D(G)$ is packing $(1, 2s_1+1, \ldots, 2s_k+1)$-colorable.
\end{prop}

They~\cite{GT1} also proved that the Petersen graph is not packing $(1,1,k,k')$-colorable for any $k,k' \ge 2$ and suggested the study of packing $(1,1,2,2)$-coloring to approach Conjecture~\ref{conjbkrw}. Bre\v sar, Klav\v zar, Rall, and Wash~\cite{BKRW2} later verified that the Petersen graph admits a packing $(1,2,3,4,5)$-coloring. Proposition~\ref{extension}  implies that if one can prove every subcubic graph except the Petersen graph is packing $(1,1,2,2)$-colorable then $\chi_p (D(G)) \le 5$ for every subcubic graph. Gastineau and Togni~\cite{GT1} also asked the question  whether the stronger claim holds: every subcubic graph except the Petersen graph is packing $(1,1,2,3)$-colorable. %\begin{prop}[Gastineau and Togni~\cite{GT1}]
%For any $k,k' \ge 2,$ the Petersen graph is not $(1,1,k,k')$-colorable.
%\end{prop}

%\begin{question}[Gastineau and Togni~\cite{GT1}]
%Is it true that every subcubic graph except the Petersen graph is packing $(1,1,2,3)$-colorable?
%\end{question}

The problem whether every subcubic graph except the Petersen graph is packing $(1,1,2,2)$-colorable is  intriguing by itself. Some subclasses of subcubic graphs were shown to have such a coloring. In particular, Bre\v sar, Klav\v zar, Rall, and Wash~\cite{BKRW2} showed that if $G$ is a generalized prism of a cycle, then $G$ is packing $(1,1,2,2)$-colorable if and only if $G$ is not the Petersen graph. Very recently, Liu, Liu, Rolek, and Yu~\cite{LLRY} proved that every subcubic graph with maximum average degree less than $\frac{30}{11}$ is $(1,1,2,2)$-colorable and thus confirmed Conjecture~\ref{conjbkrw} for subcubic graph $G$ with $mad(G)< \frac{30}{11}$. 

Many similar colorings have also been considered (e.g.~\cite{BKL2, BGT, GT1, GKT, GX1, GX2}). In particular, Gastineau and Togni~\cite{GT1} showed that subcubic graphs are packing $(1,2,2,2,2,2,2)$-colorable and packing $(1,1,2,2,2)$-colorable. They also showed that every $3$-irregular (has no adjacent vertices of degree $3$) subcubic graph is packing $(1,2,2,2)$-colorable and packing $(1,1,2)$-colorable. Gastineau~\cite{G1} showed that determining whether a subcubic bipartite graph is packing $(1, 2, 2)$-colorable and whether a subcubic graph is $(1,1,2)$-colorable are both NP-complete. In~\cite{BKL2} it was proved that subcubic graphs are packing $(1,1,2,2,3,3,k)$-colorable with color $k$ used at most once for each integer $k \ge 4$, and that $2$-degenerate subcubic graphs are packing $(1,1,2,2,3,3)$-colorable.  Borodin and Ivanova~\cite{BI} proved that every subcubic planar graph with girth at least $23$ has a packing $(2,2,2,2)$-coloring.

%\section{Main results}

Packing colorings of subclasses of subcubic outerplanar graphs was first studied by Gastineau, Holub, and Togni~\cite{GHT}, who showed upper bounds for PCN of $2$-connected subcubic outerplanar graphs with conditions on the number of internal faces. Recently, Bre\v sar, Gastineau and Togni~\cite{BGT} proved that the PCN of any $2$-connected bipartite subcubic outerplanar graph is bounded by $7$, which gives a partial answer to the question posed in several papers concerning the boundedness of the PCN in the
class of planar subcubic graphs. Moreover, they proved that every triangle-free subcubic outerplanar graph has a packing $(1,2,2,2)$-coloring (and thus a packing $(1,1,2,2)$-coloring) and their result is sharp in the sense that there exists a subcubic outerplanar graph with no triangles that is not packing $(1,2,2,3)$-colorable. They also showed that every bipartite outerplanar graph admits a packing $S$-coloring for $S = (1,3, \ldots, 3)$, where $3$ appears $\Delta$ (maximum degree) times. Their result is sharp in the sense that if one of the integers $3$ is replaced by $4$ in the sequence $S$, then there exist outerplanar bipartite graphs that do not admit a packing $S$-coloring. The following two interesting questions were also suggested by Bre\v sar, Gastineau and Togni~\cite{BGT} for future research. 

\begin{question}[Bre\v sar, Gastineau and Togni~\cite{BGT}]
Is the PCN bounded in the class of $2$-connected outerplanar subcubic graphs and is the PCN bounded in the class of $2$-connected bipartite planar subcubic graphs?
\end{question}

%\begin{question}[Bre\v sar, Gastineau and Togni~\cite{BGT}]

%\end{question}

In this paper, we prove that every $2$-connected subcubic outerplanar graph is packing $(1,1,2)$-colorable and every subcubic outerplanar graph is packing $(1,1,2,4)$-colorable. Our results are sharp in the sense that there is a $2$-connected subcubic outerplanar graph $G$ that is not $(1,1,3)$-colorable (see Example~\ref{113}) and not $(1,2,2)$-colorable (see Example~\ref{122}); and there are subcubic outerplanar graphs that are neither  $(1,2,2,4)$-colorable, nor $(1,1,3,4)$-colorable;
there are also subcubic outerplanar graphs that are  not $(1,1,2,5)$-colorable. % which is presented in Section~\ref{counterexample}.

%and thus extend the result of Bre\v sar et al~\cite{BGT} on packing $(1,1,2,2)$-coloring to all subcubic outerplanar graphs. 

\begin{theorem}\label{2connected}
Every $2$-connected subcubic outerplanar graph $G$ is packing $(1,1,2)$-colorable.
\end{theorem}

% 1) there is a $2$-connected subcubic outerplanar graph $G$ such that all of the packing $(1,1,2)$-colorings use color $2$ at a specific degree two vertex; 2)

\begin{example}\label{113}
Let $G$ be the graph obtained by starting with a four cycle $C$ such that $V(C) = \{u_1,u_2,u_3,u_4\}$ and $u_iu_{i+1} \in E(G)$ for $1 \le i \le 4$ ($5 \equiv 1$). Then we add a path $u_1v_1u_2$ of length two from $u_1$ to $u_2$ and a path $u_3v_2u_4$ from $u_3$ to $u_4$. 

%1) Assume $G$ has a packing $(1,1,2)$-coloring. Since both $u_1,u_2,v_1$ and $u_3, u_4, v_2$ form triangle, at least one vertex of $u_1,u_2,v_1$ and one vertex of $u_3, u_4, v_2$ are colored with $2$ respectively. The only way to choose two vertices from each triangle to color with $2$ is to color $v_1$ and $v_2$ with $2$, since otherwise we will pick two vertices at distance at most two. 

Assume $G$ has a packing $(1,1,3)$-coloring. Since both $u_1,u_2,v_1$ and $u_3, u_4, v_2$ form triangle, at least one vertex of $u_1,u_2,v_1$ and one vertex of $u_3, u_4, v_2$ are colored with $3$ respectively. But the diameter of $G$ is $3$, a contradiction.
\end{example} 

\begin{example}\label{122}
 {Let $G$ be the same graph as used in Example~\ref{113}. We show that $G$ is not $(1,2,2)$-colorable.}

 {Assume $G$ has a packing $(1,2,2)$-coloring. Say the colors are $1, 2_a, 2_b$. Since both $u_1,u_2,v_1$ and $u_3, u_4, v_2$ form triangles, all three colors are used exactly once on $u_1,u_2,v_1$ and $u_3, u_4, v_2$ respectively. By symmetry, we assume that $u_2$ is colored with a color in $\{2_a, 2_b\}$, say $2_a$. This contradicts the fact that $2_a$ is used on the triangle $u_3v_2u_4$. }
\end{example}

\begin{theorem}\label{1124}
Every subcubic outerplanar graph has a packing $(1,1,2,4)$-coloring $f$ such that\\
%1) If $v$ is a degree two vertex with $f(v) \in \{1_a, 1_b\}$, both neighbors are degree three vertices then $f(N(v)) \neq \{2,4\}$.\\
%1) Color $4$ is not adjacent to color $2$ within each non-trivial block (block with at least three vertices).\\
(A) Color $4$ is used at most once within each block.\\
(B) if $v$ is a vertex of degree at most $2$ and is colored with $2$, then there is no vertex of color $4$ within distance two from $v$. 
\end{theorem}

By Proposition~\ref{extension}, we also have the following immediate corollary, which confirms Conjecture~\ref{conjbkrw} for subcubic outerplanar graphs. 

\begin{cor}
If $G$ is a subcubic outerplanar graph, then $\chi_p(D(G))\le 5$. Moreover, if $H$ is a $2$-connected subcubic outerplanar graph, then $\chi_p(D(G))\le 4$.
\end{cor}

\begin{proof}
Proposition~\ref{extension} implies that if $G$ is packing $(1,1,2,2)$-colorable then $D(G)$ is packing $(1,3,3,5,5)$-colorable, which implies a packing $(1,2,3,4,5)$-coloring of $D(G)$. Similarly, since $H$ is packing $(1,1,2)$-colorable, $D(G)$ is packing $(1,3,3,5)$-colorable and thus $(1,2,3,4)$-colorable.
\end{proof}

The result of Liu, Liu, Rolek, and Yu~\cite{LLRY} implies that every subcubic planar graph with girth at least $8$ is packing $(1,1,2,2)$-colorable. We would also like to ask the following questions.

\begin{question}
Is it true that every subcubic planar graph is packing $(1,1,2,2)$-colorable?
\end{question}

\begin{question}
Is it true that every subcubic 2-connected outerplanar graph is packing $(1,2,2,2)$-colorable?
\end{question}
%The following well-known Proposition implies that if Theorem~\ref{mad2.75} is true then every subcubic planar graph with girth at least $8$ is $(1,1,2,2)$-packing colorable and thus we only need to verify the case of girth $7$ for Theorem~\ref{girth7}.

%\begin{prop}\label{girthmad}
%If $G$ is a connected planar graph with girth $g$, then $MAD(G) < 2 + \frac{4}{g-2}$.
%\end{prop}

\section{Notation and preliminaries}

We use $N_G^d(u)$ to  denote the set of all vertices that are at distance $d$ from $u$. 
We will work with {\em outerplane} graphs, that is, outerplanar graphs with a fixed drawing where all vertices are on 
the outer face.
% is a graph that has a planar drawing for which all vertices belong to the outer face of the drawing and we assume one such drawing is fixed whenever we consider an outerplanar graph. 

 A {\em block} of a graph $G$ is an inclusion maximal  subgraph with no cut vertices. By definition, each block is either $2$-connected or a $K_2$. In the former case, we call the block {\em nontrivial}. A block  in a graph $G$ is {\em pendant} if it contains at most  one cut vertex of $G$.

 Given an outerplane graph $G$,
{\em the weak dual graph,} $\mathcal{T}(G)$, is the graph that has a vertex for every bounded face of the embedding, and an edge for every pair of bounded faces sharing at least one edge. 
 Below when we say "face" we will mean an internal face, unless we explicitly say "outer face". 
For a face $F$ in an outerplane graph $G$, we denote by $C(F)$ the chordless cycle in $G$ that bounds  $F$. It is well known that 
\begin{equation}\label{op}
\mbox{\em a plane graph is outerplane if and only if its weak dual is a forest.}
\end{equation}
%The depth of a vertex $u \in T$, where $T$ is a tree with root $v$, is the distance between $u$ and $v$ in $T$.  

 By an $i$-face we will mean a face of length $i$.
 In view of~\eqref{op}, we say that an internal face $F$ in a   outerplane graph  $G$  {is} {\em pendant}, if either  $F$ corresponds to a leaf in 
 $\mathcal{T}(G)$ and $C(F)$ contains no cut vertices of $G$ or $C(F)$ induces a pendant block in $G$.

\begin{claim}\label{3face}
Each  $3$-face in a $2$-connected  {subcubic} outerplane graph is pendant.
\end{claim}

\begin{proof} Let $F$ be a $3$-face with $C(F)=xyzx$ in a $2$-connected  {subcubic} outerplane graph $G$. If, say $d(x)=2$, then the edges
$xy$ and $xz$ are on the boundary of the outer face, and so $F$ is pendant. 

So, suppose $d(x)=d(y)=d(z)=3$. Let $x'$, $y'$ and $z'$ be the neighbors of $x,y$ and $z$ respectively outside of $\{x,y,z\}$ (some of them can coincide). Since $G$ is $2$-connected, all $x'$, $y'$ and $z'$ are in the same component of $G-\{x,y,z\}$. But then $G$ contains a $K_4$-minor, which implies that $G$ is not outerplane, a contradiction.
\end{proof}
%We define the special block graph $\mathcal{B}(G)$ of $G$ to be the graph with each non-trivial block (blocks with at least three vertices) $b$ contract to one vertex $v_b$ (if $v$ was joining to any vertex in $b$, then $v$ has an edge to $v_b$). Obviously, for a connected outerplanar graph $G$, $\mathcal{B}(G)$ is a tree.

\section{Proof of Theorem~\ref{2connected}}

All our $(1,1,2)$-colorings will use the colors $1_a,1_b$ and $2$.

%\begin{theorem}\label{2connected}
%For every $2$-connected outerplanar subcubic graph $G$, and for every degree three vertex $v \in V(G)$, there is a packing $(1,1,2)$-coloring $f$ such that $f(v) \neq 2$. 
%\end{theorem}

%\begin{proof}
\textbf{Proof.} Suppose the theorem is false and $G$ is a smallest  $2$-connected  {subcubic} outerplane graph that has no
$(1,1,2)$-coloring. Let $n=|V(G)|$. Then $n\geq 4$, since otherwise we can color all vertices of $G$ with different colors.

\begin{claim}\label{pendant}
Each pendant  face in $G$ is a $3$-face.
\end{claim}

\begin{proof} Suppose $F$ is a pendant face in $G$ with $C(F)=u_1u_2\ldots u_ku_1$ where $k\geq 4$. If
$V(G)=\{ u_1,u_2,\ldots ,u_k\}$ then we color $u_1$ with $2$ and the remaining vertices alternately with $1_a$ and $1_b$.
So suppose $d(u_1)=d(u_k)=3$ and $d(u_2)=d(u_3)=\ldots=d(u_{k-1})=2$.
 Let $G'=G-\{u_2,\ldots, u_{k-1}\}$. 
Then $G'$ is a $2$-connected outerplane graph.  By the minimality of $G$, $G'$ has  a packing $(1,1,2)$-coloring $f$. 

We can extend $f$ to  the vertices $u_2,\ldots, u_{k-1}$ by coloring them alternately with $1_a$ and $1_b$, unless $k$ is odd and
$\{f(u_1),f(u_k)\}=\{1_a,1_b\}$. In this exceptional case, assuming $f(u_1)=1_a$, $f(u_k)=1_b$, we let $f(u_2)=1_b$, $f(u_3)=2$, and color
$u_4,u_5,\ldots,u_{k-1}$ alternately with $1_a$ and $1_b$.
 Thus {when $k \ge 4$}, in all cases  , we get
a packing $(1,1,2)$-coloring of $G$, a contradiction.
\end{proof}

%For a chosen degree three vertex $v \in V(G)$, we call a packing $(1,1,2)$-coloring $f$ on $G$ such that $f(v) \neq 2$ {\em good} with respect to $v$.
\begin{claim}\label{noeven}
$G$ has no  face of even length.
\end{claim}

\begin{proof} Suppose
 $G$ has a face $F_0$ with  $C(F_0)= u_1u_2 \ldots u_ku_1$, where $k$ is even. Let $G_1,\ldots,G_\ell$ be the components of
 $G-\{u_1,u_2 \ldots u_k\}$. Since $G$ is   $2$-connected and outerplane, each $G_i$ has exactly two neighbors on $C(F_0)$,
 and these neighbors are consecutive on $C(F_0)$. Let these neighbors be $u_{j(i)}$ and $u_{j(i)+1}$, and let the neighbors of
 $u_{j(i)}$ and $u_{j(i)+1}$ in $V(G_i)$ be $v_i$ and $v'_i$ (possibly, $v'_i=v_i$). If $v'_i\neq v_i$ and $v_iv'_{i}\notin E(G)$ then we define $G'_i=G_i+v_iv'_i$, otherwise we let $G'_i=G_i$.
 
 With these definitions, if $v_i'=v_i$  {then,} by Claim~\ref{3face}, $G'_i$ is a single vertex; otherwise, $G'_i$ is 
 a $2$-connected outerplane graph. So by the minimality of $G$, each $G'_i$ has  a packing $(1,1,2)$-coloring $f'_i$
 such that 
 \begin{equation}\label{dif}
\mbox{\em  if $v'_i\neq v_i$, then $f'_i(v'_i)\neq f'_i(v_i)$.}
 \end{equation}
 
 We now define a  $(1,1,2)$-coloring $f$ of $G$ as follows:
 
(i)  For $1 \le j \le \frac{k}{2}$, we let
  $f(u_{2j-1}) = 1_a$ and  $f(u_{2j}) = 1_b$.

(ii) If $1 \le i \le \ell$ and $v'_i=v_i$, then $G'_i$ is a single vertex $v_i$ and we let $f(v_i)=2$.

(iii) If  $1 \le i \le \ell$ and $v'_i\neq v_i$, then by~\eqref{dif} and the fact that $\{f(u_{j(i)}),f(u_{j(i)+1})\}=\{1_a,1_b\}$, we can switch
the names of the colors $1_a$ and $1_b$ in $f'_i$  so that $f'_i(v_i)\neq f(u_{j(i)}) $ and $f'_i(v'_i)\neq f(u_{j(i)+1})$. In this case,
we let $f(v)=f'_i(v)$ for each $v\in V(G_i)$.

By construction, $f$ is a packing $(1,1,2)$-coloring on $G$, since the vertices of color $2$ in different $G_i$ are at distance at least $3$.
\end{proof}

 If $G$ has only one face apart from the outer face, then $G$ is an odd cycle, say $u_1u_2 \ldots u_{2k+1}u_1$, and we can color 
 its vertices apart from $u_{2k+1}$ alternately with $1_a$ and $1_b$ and let
 $f(u_{2k+1})=2$. Thus,  suppose $G$ has at least two faces. Let $F_0$ be a pendant face  corresponding to an end vertex in a longest path in $\mathcal{T}(G)$. By  Claim~\ref{pendant}, $C(F_0)$ is a $3$-cycle.
 
 Let $F_0'$ be the unique face adjacent to $F_0$. By our choice of $F_0$, 
 \begin{equation}\label{only}
\mbox{\em  $F_0'$ is adjacent to at most one non-pendant face.}
 \end{equation}
 
If $|C(F_0')|=3$,  {then,} by Claim~\ref{3face},     $G=K_4-e$. In this case, we color the two vertices of degree $2$ in $G$ with $1_a$
and the remaining two vertices with $1_b$ and $2$.
 
Thus, we may assume $F_0'$ is a face with an odd length $k\geq 5$. Since $F_0'$ is an odd face and each face adjacent to $F_0'$ (apart from the outer face) shares exactly two vertices with $F_0'$,  at least one vertex in $C(F_0')$ has degree two in $G$.
Fix one such vertex, say $w_1$. Let $C(F_0') = w_1w_2 \ldots w_kw_1$.

%We first fix the color of vertices in the cycle by each $f(u_{2i-1}) = 1_a$ and each $f(u_{2i}) = 1_b$, where $1 \le i \le \frac{k}{2}$. Thus, if $v \in V(C(F_0))$, then $f(v) \neq 2$. We assume $v \notin V(C(F_0))$.
%, then $v \notin \{w_1, w_2\}$ \textbf{Case a:} $v \in \{w_3, w_k\}$. Say $v = w_3$. $w_3$ is a degree three vertex and
If one of $w_2, w_k$ has degree two, say $w_2$, then  we delete $w_1, w_2$ and add the edge $w_3w_k$ (it is not in $E(G)$ since $F_0'$ has length at least five). This results in a $2$-connected subcubic outerplane graph $G'$ with fewer vertices. By the minimality of $G$,  $G'$
has a packing $(1,1,2)$-coloring $f'$. If $2 \notin \{f'(w_3), f'(w_k)\}$, say $f'(w_3) = 1_a$ and $f'(w_k) = 1_b$, then we color $w_1, w_2$ with $1_a, 1_b$. If $2 \in \{f'(w_3), f'(w_k)\}$, say $f'(w_3) = 2$ and $f'(w_k) = 1_b$, then we color $w_1, w_2$ with $1_a, 1_b$.% \textbf{Case b:} $v \notin \{w_3, w_k\}$. Then we delete $w_1, w_2$ and add the edge $w_3w_k$ to obtain a $2$-connected subcubic outerplanar graph $G'$. Since $v \notin \{w_3, w_k\}$, $v$ is a degree three vertex in $G'$ and we apply induction hypothesis to obtain a packing $(1,1,2)$-coloring $f'$ such that $f'(v) \neq 2$. If $2 \in \{f'(w_3), f'(w_k)\}$, say $f'(w_3) = 2$ and $f'(w_k) = 1_a$, then we color $w_1, w_2$ with $1_b, 1_a$; if $2 \notin \{f'(w_3), f'(w_k)\}$, say $f'(w_3) = 1_a$ and $f'(w_k) = 1_b$, then we color $w_1, w_2$ with $1_a, 1_b$.

Thus, we may assume that both neighbors of $w_1$, i.e., $w_2$ and $w_k$, have degree three. Let $G_1,\ldots,G_\ell$ be the components of  $G-\{w_1,w_2 \ldots w_k\}$. Since $k\geq 5$ and $d(w_2)=d(w_k)=3$, $\ell\geq 2$.
As in the proof of Claim~\ref{noeven},
 each $G_i$ has exactly two neighbors on $C(F_0)$,
 and these neighbors are consecutive on $C(F_0)$. Let these neighbors be $w_{j(i)}$ and $w_{j(i)+1}$, and let the neighbors of
 $w_{j(i)}$ and $w_{j(i)+1}$ in $V(G_i)$ be $v_i$ and $v'_i$ (possibly, $v'_i=v_i$). We can rename $G_i$s so that 
 $j(1)<j(2)<\ldots<j(\ell)$.
 By~\eqref{only} and Claim~\ref{pendant}, at most one
 of $G_1,\ldots,G_\ell$ is not a single vertex. 
 By the  symmetry between $w_2$ and $w_k$, we may assume that $G_1$ is a single vertex.
 
 We start coloring by letting $f(w_2)=2$ and coloring the remaining vertices of $C(F'_0)$ alternately with $1_a$ and $1_b$.
 Then let $f(v_1)=1_b$ and color the unique vertex in  each other single-vertex $G_i$ with $2$. If $G-\{w_1,w_2 \ldots w_k\}$
 has no larger components, then we are done. Otherwise, suppose $G_{i_0}$ is the unique "large" component of 
 $G-\{w_1,w_2 \ldots w_k\}$.
 If  $v_{i_0}v'_{i_0}\notin E(G)$, then we define $G'_{i_0}=G_{i_0}+v_{i_0}v'_{i_0}$, otherwise we let $G'_{i_0}=G_{i_0}$.
By the minimality of $G$, $G'_{i_0}$ has  a packing $(1,1,2)$-coloring $f'$
 such that  $f'(v'_{i_0})\neq f'(v_{i_0})$. As in the proof of Claim~\ref{noeven},
  the facts that $\{f(w_{j(i_0)}),f(w_{j(i_0)+1})\}=\{1_a,1_b\}$ and $f'(v'_{i_0})\neq f'(v_{i_0})$, we can switch
the names of the colors $1_a$ and $1_b$ in $f'$  so that $f'(v_{i_0})\neq f(w_{j(i_0)}) $ and $f'(v'_{i_0})\neq f(w_{j(i_0)+1})$. 
After that, we let $f(v)=f'(v)$ for each $v\in V(G_{i_0})$.

So, we obtain a packing $(1,1,2)$-coloring of $G$. This contradicts the choice of $G$ and proves the theorem.

%%%%%%%%%%%%%%%%%%%%%%%%%%%%%%%%%%%%%%%%%%%%%%%%%%%%%%%%%%%%%%%%%%%%%%%%%%%

%\end{proof}

\section{Proof of Theorem~\ref{1124}}
%\begin{theorem}
%For every subcubic outerplanar graph $G$, there is a packing $(1,1,2,4)$-coloring $f$ such that\\
%1) If $v$ is a degree two vertex with $f(v) \in \{1_a, 1_b\}$, both neighbors are degree three vertices then $f(N(v)) \neq \{2,4\}$.\\
%2) Color $4$ is used at most once within each non-trivial block.\\
%3) If $v$ is a degree two vertex colored with $2$, then there is no vertex colored with $4$ within distance two.
%\end{theorem}
%, and all vertices in $N^2(v)$ are in the same block, 

%\begin{proof}

By a {\em feasible} coloring of $G$ we call a coloring of $G$ with colors $1_a,1_b,2,4$ such that the distance between two vertices of color $i_x$ is at least $i+1$ for all $i \in \{1, 2, 4\}$ and $x\in \{a,b\}$, and $f$ satisfies conditions (A) and (B) of Theorem~\ref{1124}.

%\textbf{Proof.} 
Suppose, Theorem~\ref{1124} fails and $G$ is a smallest outerplane graph not admitting a feasible coloring. Clearly,
\begin{equation}\label{2con}
\mbox{\em  $G$ is connected and $\delta(G)\geq 2$.}
 \end{equation}

It follows that every pendant block is nontrivial. 
So if $G$  has only one non-trivial block, then  it has no other blocks.
In this case, $G$ has  a packing $(1,1,2)$-coloring by Theorem~\ref{2connected}.
 Hence we may assume that  $G$ has at least two  blocks, and thus at least two pendant blocks (which are nontrivial).

\begin{claim}\label{pendant2}
Each pendant  face in $G$ is a $3$-face.
\end{claim}

\begin{proof} Suppose $F$ is a pendant face with $C(F)=u_1u_2\ldots u_ku_1$ where $k\geq 4$. 
If $F$ does not contain a cut vertex of $G$, then we can repeat the proof of Claim~\ref{pendant} word by word. Note that when we color $u_3$ with $2$ in the second paragraph of Claim~\ref{pendant}, it is at distance at least two from $u_1$ and $u_k$ respectively, and thus condition (B) in Theorem~\ref{1124} is satisfied.
So, suppose $u_1$ is a cut vertex of $G$, and its neighbor outside of $C(F)$ is $v$. Recall that by the definition of pendant faces,
in this case $C(F)$ induces a pendant block in $G$. 

By the minimality of $G$, $G-\{u_1,u_2,\ldots ,u_k\}$ has  a feasible coloring $f$.
We  extend $f$ to  {$G$} as follows. First choose {$f(u_1)\in \{1_a,1_b\}-f(v)$}. By symmetry, assume $f(u_1)=1_b$.
% and $f(u_2)\in \{1_a,1_b\}-f(u_1)$.
 If $k$ is even, then
we can color
$u_2,\ldots,u_k$ alternately with $1_a$ and $1_b$. Otherwise, let $f(u_2)=1_a$, $f(u_3)=2$ and color
$u_4,\ldots,u_k$ alternately with $1_b$ and $1_a$. In both cases, we obtain
a feasible coloring of $G$, a contradiction.
\end{proof}

 Let $G_0$ be one of the pendant blocks.
   Let the cut edge connecting $G_0$ and $G-G_0$ be $u_1v_1$ with $v_1 \in V(G_0)$. Let $F_0$ be the face in $G_0$ containing $v_1$
  with $C(F_0)=v_1v_2\ldots v_kv_1$.  Let $N(u_1) = \{u_2, u_3, v_1\}$. By the minimality of $G$, graph  $G'=G-G_0$ has a feasible coloring $f$.

\textbf{Case 1:} $G_0$ is a cycle. By Claim~\ref{pendant2}, $k=3$.

\textbf{Case 1.1:} $f(u_1) \in \{1_a, 1_b\}$, say $f(u_1) = 1_a$. We let $f(v_1)=1_b$, $f(v_2)=1_a$ and $f(v_3)=2$.

\textbf{Case 1.2:} $f(u_1) = 2$. Since $f$ is a feasible coloring of $G'$ and $d_{G'}(u_1)\leq 2$, by (B) in the claim of Theorem~\ref{1124},
{there is no vertex in $G-G_0$ colored with $4$ within distance two from $u_1$}. Then we  let $f(v_1)=1_a$, $f(v_2)=1_b$ and $f(v_3)=4$.

\textbf{Case 1.3:} $f(u_1) = 4$. If $2 \in \{f(u_2), f(u_3)\}$, say $f(u_2) = 2$ and $f(u_3) = 1_a$, then we recolor $u_1$ with $1_b$ and obtain Case 1.1. Thus we may assume that $\{f(u_2), f(u_3)\} = \{1_a, 1_b\}$. In this case,
Then we let $f(v_1)= 2$, $f(v_2)=1_b$ and $f(v_3)= 1_a$.

\textbf{Case 2:} $F_0$ is  adjacent in $\mathcal{T}(G_0)$ only to pendant faces. Let these faces be $F_1, \ldots F_{\ell}$ ordered so   that
the indices of the vertices in $V(F_i) \cap V(F_0)$ are larger than the indices of the vertices in $V(F_j) \cap V(F_0)$ if and only if $i > j$.
Suppose the common vertices of $C(F_0)$ and $C(F_1)$ are $v_p$ and $v_{p+1}$.
By Claim~\ref{pendant2}, each $F_i$ is a $3$-face. If $k$ is even, we can choose $f(v_1)\in \{1_a,1_b\}-f(u_1)$, then color alternately with 
$1_a$ and $1_b$ all vertices $v_2,\ldots,v_k$, and for each $i=1,\ldots,\ell$, color the unique vertex $w_i\in C(F_i)-C(F_0)$ with $2$. So we may
assume $k$ is odd.

\textbf{Case 2.1:} $f(u_1) \neq 2$. Let $f(v_p)=2$ and color alternately    with 
$1_a$ and $1_b$ all vertices  in $V(F_0)-v_p$ so that
 $f(v_1)\neq f(u_1)$. For all $2\leq i\leq \ell$, color $w_i$ with $2$ and choose  $f(w_1)\in \{1_a,1_b\}-f(v_{p+1})$.

\textbf{Case 2.2:} $f(u_1) = 2$. As in Case 1.2,  within distance two of $u_1$
 there is no vertex in $G-G_0$ colored with $4$. 
 We color vertices in $G_0$ almost as in Case 2.1, except we color $v_p$ with $4$. Observe that conditions (A) and  (B) in the claim of Theorem~\ref{1124}
hold for the new coloring, and that $v_p$ is at distance at least $2$ from $u_1$ which in turn is at distance at least {$5$} from other vertices of color $4$.

\textbf{Case 3:} $F_0$ is adjacent to some non-pendant face. Let $R$ be a pendant face of $G_0$ that has the largest distance from $F_0$ in the weak dual of $G_0$. By the description of Case 3, this distance is at least two. Let $R_0$ be the face that $R$ is adjacent to. By the choice of $R$,  $R_0$ is adjacent to only one non-pendant face, say $R_0'$. 

Let $C(R_0)=x_1x_2\ldots x_rx_1$ and
$V(R_0) \cap V(R_0') = \{x_1, x_2\}$. Let $R_1, \ldots, R_m$ be the pendant faces that are adjacent to $R_0$ arranged in the order of  $C(R_0)$.
Recall that by Claim~\ref{pendant2}, each pendant $R_i$ is a $3$-face.
Assume that for $i=1,\ldots,m$,  $V(R_i) \cap V(R_0) = \{x_{q_i}, x_{q_i + 1}\}$ and $V(R_i) \setminus V(R_0) = \{y_i\}$.

\textbf{Case 3.1:} $R_0'$  has only three vertices; say the common neighbor of $x_1$ and $x_2$  {in $R_0'$} is $x_0$. Then by our construction, $\bigcup_{i=0}^mV(R_i)\cup \{x_0\}$
comprises $V(G_0)$, and
 $x_0$ is the vertex in $G_0$ that is adjacent to $u_1$. By the minimality of $G$,  $G'$ has a feasible coloring $f$. Recall that the neighbor of $x_0$
 in $V(G-G_0)$ is $u_1$.

\textbf{Case 3.1.1:} $f(u_1) \in \{1_a, 1_b\}$.  Then we color $x_1$ with $2$ and  $x_0, x_2,x_3,\ldots,x_r$ alternately with $1_a$ and $1_b$
so that $f(x_0)\neq f(u_1)$. After that, we let $f(y_m)=4$, and $f(y_i)=2$ for all $1\leq i\leq m-1$. Then the coloring will be a packing
$(1,1,2,4)$-coloring and  the conditions (A) and (B) will hold.

\textbf{Case 3.1.2:} $f(u_1) = 2$. Since $d_{G'}(u_1) \le 2$, the distance from $u_1$ to a vertex of color $4$ in $G'$ is at least $3$.
Then we color $x_1$ with $4$ and color $x_{q_m+1}$ with $2$,  $x_0, x_2,x_3,\ldots,x_{q_m}, y_m$ alternately with $1_a$ and $1_b$, $x_{q_m+2}, \ldots, x_r$ (if $q_m+1 < r$) alternatively with $1_a$ and $1_b$.
 After that, if $m \ge 2$ then we let $f(y_i)=2$ for all $1\leq i\leq m-1$.

\textbf{Case 3.1.3:} $f(u_1) = 4$. If $\{f(u_2),f(u_3)\}\neq \{1_a,1_b\}$, then we can recolor $u_1$ with a color in 
$ \{1_a,1_b\}\setminus \{f(u_2),f(u_3)\}$ and get Case 3.1.1.
Otherwise, $2 \notin \{f(u_2), f(u_3)\}$. By Theorem~\ref{2connected}, $G_0$ has a packing $(1,1,2)$-coloring $f_0$.
Since the vertices $x_0,x_1,x_2$ have degree $3$ in $G$, the coloring $f\cup f_0$ will be a  packing $(1,1,2,4)$-coloring of $G$
satisfying (A) and (B).

\textbf{Case 3.2:} $R_0'$ has at least four vertices. For $j=1,2$, let $x'_j$  be the neighbor of $x_j$ on $C(R_0')$ distinct from $x_{3-j}$.
By the case, $x'_2 \neq x'_1$. Let $G''$ be obtained from $G-\bigcup_{i=0}^mV(R_i)$ by adding edge $x'_1x'_2$ if this edge is not in $G$.
 By the minimality of $G$,
the subcubic outerplane graph $G''$ has a feasible coloring $f$. Since $x'_1x'_2\in E(G'')$, by symmetry, we may assume $f(x'_2)\neq 2$.
We color $x_{q_1}$ with $2$ and the remaining vertices of $R_0$ alternately with 
$1_a$ and $1_b$ so that $f(x_1)\neq f(x'_1)$ and hence $f(x_2)\neq f(x'_2)$. We can provide the two last inequalities because if
 $f(x'_1), f(x'_2)\in \{1_a, 1_b\}$, then  $f(x'_1)\neq f(x'_2)$. 
 
 After that, we choose $f(y_1)\in \{1_a, 1_b\}-f(x_{q_1+1})$ and let
$f(y_j)=2$ for $j=2,3,\ldots,m$. We obtain a feasible coloring of $G$, a contradiction. This proves Theorem~\ref{1124}.

\begin{figure}[ht]\label{f1}
\hspace{5mm}
\begin{center}
%\begin{figure}
\begin{tikzpicture}[scale = 0.379, transform shape]

\node [ shape=circle, minimum size=0.1cm,  fill = black!1000, align=center] (v3) at (-11.5,6.5) {};
\node [ shape=circle, minimum size=0.1cm,  fill = black!1000, align=center] (v4) at (-10,5) {};
\node [ shape=circle, minimum size=0.1cm,  fill = black!1000, align=center] (v5) at (-7,6.5) {};
\node [ shape=circle, minimum size=0.1cm,  fill = black!1000, align=center] (v6) at (-10,6.5) {};

\node [ shape=circle, minimum size=0.1cm,  fill = black!1000, align=center] (v9) at (-4,5) {};
\node [ shape=circle, minimum size=0.1cm,  fill = black!1000, align=center] (v10) at (-5.5,3.5) {};
\node [ shape=circle, minimum size=0.1cm,  fill = black!1000, align=center] (v11) at (-5.5,6.5) {};
\node [ shape=circle, minimum size=0.1cm,  fill = black!1000, align=center] (v12) at (-7,5) {};

\node [ shape=circle, minimum size=0.1cm,  fill = black!1000, align=center] (v22) at (-4,3.5) {};
\node [ shape=circle, minimum size=0.1cm,  fill = black!1000, align=center] (v23) at (-13,3.5) {};
\node [ shape=circle, minimum size=0.1cm,  fill = black!1000, align=center] (v24) at (-13,5) {};
\node [ shape=circle, minimum size=0.1cm,  fill = black!1000, align=center] (v25) at (-11.5,3.5) {};

\node [ shape=circle, minimum size=0.1cm,  fill = black!1000, align=center] (v1) at (-8.5,8) {};
\node [ shape=circle, minimum size=0.1cm,  fill = black!1000, align=center] (v2) at (-8.5,10) {};

\node at (-7,7.15) {\huge{$y_1$}};
\node at (-7.65,5) {\huge{$y_2$}};
\node at (-6,3) {\huge{$y_4$}};
\node at (-3.25,3.5) {\huge{$y_6$}};
\node at (-3.5,5.5) {\huge{$y_5$}};
\node at (-5.5,7.15) {\huge{$y_3$}};

\node at (-2,0.5) {\huge{$z_1$}};
\node at (-4,-0.5) {\huge{$z_2$}};
\node at (-1,-0.5) {\huge{$z_3$}};
\node at (-2,-4.5) {\huge{$z_6$}};
\node at (-1,-3.5) {\huge{$z_5$}};
\node at (-4,-3.5) {\huge{$z_4$}};

\node at (-7.75,10) {\huge{$x_1$}};

\node at (-7.5,8) {\huge{$x_4$}};

\draw  (v2) edge (v1);
\draw  (v1) edge (v6);
\draw  (v1) edge (v5);
\draw  (v6) edge (v3);
\draw  (v3) edge (v24);
\draw  (v24) edge (v23);
\draw  (v23) edge (v25);
\draw  (v25) edge (v4);
\draw  (v4) edge (v6);
\draw  (v3) edge (v4);
\draw  (v24) edge (v25);
\draw  (v5) edge (v11);
\draw  (v5) edge (v12);
\draw  (v12) edge (v10);
\draw  (v10) edge (v22);
\draw  (v22) edge (v9);
\draw  (v9) edge (v11);
\draw  (v11) edge (v12);
\draw  (v10) edge (v9);

\node [ shape=circle, minimum size=0.1cm,  fill = black!1000, align=center] (v7) at (-14.5,2) {};
\node [ shape=circle, minimum size=0.1cm,  fill = black!1000, align=center] (v8) at (-14.5,0) {};
\node [ shape=circle, minimum size=0.1cm,  fill = black!1000, align=center] (v13) at (-16.5,2) {};
\node [ shape=circle, minimum size=0.1cm,  fill = black!1000, align=center] (v31) at (-15.5,-1) {};
\node [ shape=circle, minimum size=0.1cm,  fill = black!1000, align=center] (v35) at (-13.5,-1) {};
\node [ shape=circle, minimum size=0.1cm,  fill = black!1000, align=center] (v34) at (-13.5,-3) {};
\node [ shape=circle, minimum size=0.1cm,  fill = black!1000, align=center] (v32) at (-15.5,-3) {};
\node [ shape=circle, minimum size=0.1cm,  fill = black!1000, align=center] (v33) at (-14.5,-4) {};
\node [ shape=circle, minimum size=0.1cm,  fill = black!1000, align=center] (v40) at (-17.5,1) {};
\node [ shape=circle, minimum size=0.1cm,  fill = black!1000, align=center] (v36) at (-17.5,3) {};
\node [ shape=circle, minimum size=0.1cm,  fill = black!1000, align=center] (v37) at (-19.5,3) {};
\node [ shape=circle, minimum size=0.1cm,  fill = black!1000, align=center] (v39) at (-19.5,1) {};
\node [ shape=circle, minimum size=0.1cm,  fill = black!1000, align=center] (v38) at (-20.5,2) {};
\node [ shape=circle, minimum size=0.1cm,  fill = black!1000, align=center] (v14) at (-2.5,2) {};
\node [ shape=circle, minimum size=0.1cm,  fill = black!1000, align=center] (v15) at (-2.5,0) {};
\node [ shape=circle, minimum size=0.1cm,  fill = black!1000, align=center] (v16) at (-0.5,2) {};
\node [ shape=circle, minimum size=0.1cm,  fill = black!1000, align=center] (v26) at (-3.5,-1) {};
\node [ shape=circle, minimum size=0.1cm,  fill = black!1000, align=center] (v27) at (-1.5,-1) {};
\node [ shape=circle, minimum size=0.1cm,  fill = black!1000, align=center] (v30) at (-3.5,-3) {};
\node [ shape=circle, minimum size=0.1cm,  fill = black!1000, align=center] (v28) at (-1.5,-3) {};
\node [ shape=circle, minimum size=0.1cm,  fill = black!1000, align=center] (v29) at (-2.5,-4) {};
\node [ shape=circle, minimum size=0.1cm,  fill = black!1000, align=center] (v17) at (0.5,3) {};
\node [ shape=circle, minimum size=0.1cm,  fill = black!1000, align=center] (v21) at (0.5,1) {};
\node [ shape=circle, minimum size=0.1cm,  fill = black!1000, align=center] (v18) at (2.5,3) {};
\node [ shape=circle, minimum size=0.1cm,  fill = black!1000, align=center] (v20) at (2.5,1) {};
\node [ shape=circle, minimum size=0.1cm,  fill = black!1000, align=center] (v19) at (3.5,2) {};
\draw  (v23) edge (v7);
\draw  (v7) edge (v8);
\draw  (v7) edge (v13);
\draw  (v22) edge (v14);
\draw  (v14) edge (v15);
\draw  (v14) edge (v16);
\draw  (v16) edge (v17);
\draw  (v17) edge (v18);
\draw  (v18) edge (v19);
\draw  (v19) edge (v20);
\draw  (v20) edge (v21);
\draw  (v21) edge (v16);
\draw  (v17) edge (v21);
\draw  (v18) edge (v20);
\draw  (v15) edge (v26);
\draw  (v26) edge (v27);
\draw  (v15) edge (v27);
\draw  (v27) edge (v28);
\draw  (v28) edge (v29);
\draw  (v29) edge (v30);
\draw  (v30) edge (v28);
\draw  (v30) edge (v26);
\draw  (v8) edge (v31);
\draw  (v31) edge (v32);
\draw  (v32) edge (v33);
\draw  (v33) edge (v34);
\draw  (v34) edge (v35);
\draw  (v35) edge (v8);
\draw  (v31) edge (v35);
\draw  (v34) edge (v32);
\draw  (v13) edge (v36);
\draw  (v36) edge (v37);
\draw  (v37) edge (v38);
\draw  (v38) edge (v39);
\draw  (v39) edge (v40);
\draw  (v40) edge (v13);
\draw  (v36) edge (v40);
\draw  (v39) edge (v37);

\node (v47) at (-14.5,-5.5) {};
\node (v46) at (-2.5,-5.5) {};
\node (v45) at (5,2) {};
\node (v48) at (-22,2) {};

\draw  (v19) edge (v45);
\draw  (v29) edge (v46);
\draw  (v33) edge (v47);
\draw  (v38) edge (v48);
\draw  (-16.5,-5.5) rectangle (-12.5,-7.5);
\draw  (-4.5,-5.5) rectangle (-0.5,-7.5);
\draw  (-22,0) rectangle (-24,4);
\draw  (5,4) rectangle (7,0);

\node at (-14.5,-6.5) {\huge{$G_1$}};
\node at (-23,2) {\huge{$G_1$}};
\node at (-2.5,-6.5) {\huge{$G_1$}};
\node at (6,2) {\huge{$G_1$}};
\draw  (-25,9) rectangle (8,-9.5);
\node at (-8.5,-1) {\Huge{$G_2$}};

\node at (-11.5,7.25) {\huge{$t_2$}};
\node at (-9.35,5) {\huge{$t_3$}};
\node at (-10,7.25) {\huge{$t_1$}};

\node at (0.5,0.35) {\huge{$s_3$}};
\node at (-0.5,2.75) {\huge{$s_1$}};
\node at (0.5,3.65) {\huge{$s_2$}};

\end{tikzpicture}

\caption{Gadget $G_2$}\label{g2}

\end{center}

\vspace{1cm}

\begin{minipage}[b]{0.5\textwidth}
\hspace{10mm}
\begin{tikzpicture}[scale = 0.5, transform shape]
\node [ shape=circle, minimum size=0.1cm,  fill = black!1000, align=center] (v3) at (-11.5,6.5) {};
\node [ shape=circle, minimum size=0.1cm,  fill = black!1000, align=center] (v4) at (-10,5) {};
\node [ shape=circle, minimum size=0.1cm,  fill = black!1000, align=center] (v5) at (-7,6.5) {};
\node [ shape=circle, minimum size=0.1cm,  fill = black!1000, align=center] (v6) at (-10,6.5) {};

\node [ shape=circle, minimum size=0.1cm,  fill = black!1000, align=center] (v9) at (-4,5) {};
\node [ shape=circle, minimum size=0.1cm,  fill = black!1000, align=center] (v10) at (-5.5,3.5) {};
\node [ shape=circle, minimum size=0.1cm,  fill = black!1000, align=center] (v11) at (-5.5,6.5) {};
\node [ shape=circle, minimum size=0.1cm,  fill = black!1000, align=center] (v12) at (-7,5) {};

\node [ shape=circle, minimum size=0.1cm,  fill = black!1000, align=center] (v22) at (-4,3.5) {};
\node [ shape=circle, minimum size=0.1cm,  fill = black!1000, align=center] (v23) at (-13,3.5) {};
\node [ shape=circle, minimum size=0.1cm,  fill = black!1000, align=center] (v24) at (-13,5) {};
\node [ shape=circle, minimum size=0.1cm,  fill = black!1000, align=center] (v25) at (-11.5,3.5) {};

\node [ shape=circle, minimum size=0.1cm,  fill = black!1000, align=center] (v1) at (-8.5,8) {};
\node [ shape=circle, minimum size=0.1cm,  fill = black!1000, align=center] (v2) at (-8.5,10) {};

\node at (-7,7.15) {\huge{$v_1$}};
\node at (-7.65,5) {\huge{$v_2$}};
\node at (-6,3) {\huge{$v_4$}};
\node at (-3.5,3) {\huge{$v_6$}};
\node at (-3.5,5.5) {\huge{$v_5$}};
\node at (-5.5,7.15) {\huge{$v_3$}};

\node at (-10,7.15) {\huge{$u_1$}};
\node at (-11.5,7.15) {\huge{$u_2$}};
\node at (-9.35,5) {\huge{$u_3$}};
\node at (-13.5,3) {\huge{$u_6$}};
\node at (-11,3) {\huge{$u_5$}};
\node at (-13.5,5.5) {\huge{$u_4$}};

\node at (-7.75,10) {\huge{$z_6$}};
\node at (-7.75,8) {\huge{$w_1$}};

\draw  (v2) edge (v1);
\draw  (v1) edge (v6);
\draw  (v1) edge (v5);
\draw  (v6) edge (v3);
\draw  (v3) edge (v24);
\draw  (v24) edge (v23);
\draw  (v23) edge (v25);
\draw  (v25) edge (v4);
\draw  (v4) edge (v6);
\draw  (v3) edge (v4);
\draw  (v24) edge (v25);
\draw  (v5) edge (v11);
\draw  (v5) edge (v12);
\draw  (v12) edge (v10);
\draw  (v10) edge (v22);
\draw  (v22) edge (v9);
\draw  (v9) edge (v11);
\draw  (v11) edge (v12);
\draw  (v10) edge (v9);
\draw  (-14.5,9) rectangle (-2.5,2);

\node at (-8.5,3.5) {\huge{$G_1$}};
\end{tikzpicture}
\caption{Gadget $G_1$.}\label{g1}
\end{minipage}
\hspace{15mm}
\begin{minipage}[b]{0.4\textwidth}
\begin{tikzpicture}[scale=0.5, transform shape]

\node [ shape=circle, minimum size=0.1cm,  fill = black!1000, align=center] (v1) at (-11.5,6.5) {};
\node [ shape=circle, minimum size=0.1cm,  fill = black!1000, align=center] (v2) at (-12.5,4.5) {};
\node [ shape=circle, minimum size=0.1cm,  fill = black!1000, align=center] (v3) at (-10.5,4.5) {};

\draw  (v1) edge (v2);
\draw  (v2) edge (v3);
\draw  (v3) edge (v1);
\node (v4) at (-11.5,8) {};
\node (v5) at (-9,3.5) {};
\node (v6) at (-14,3.5) {};
\draw  (v1) edge (v4);
\draw  (v3) edge (v5);
\draw  (v2) edge (v6);
\draw  (-13.5,10) rectangle (-9.5,8);
\draw  (-9,4) rectangle (-5,2);
\draw  (-18,4) rectangle (-14,2);
\node at (-11.5,9) {\huge{$G_2$}};
\node at (-16,3) {\huge{$G_2$}};
\node at (-7,3) {\huge{$G_2$}};
\node at (-10,5) {\huge{$x_1$}};
\node at (-13,5) {\huge{$x_2$}};
\node at (-12,7) {\huge{$x_3$}};
\node at (-11.5,3) {\huge{$G$}};
\end{tikzpicture}

\caption{The construction.}\label{g}
\end{minipage}

\end{figure}

\section{Sharpness of the bound}\label{counterexample}
 {To show the sharpness of our result that every subcubic outerplanar graph is $(1,1,2,4)$-colorable, we show that there are subcubic outerplanar graphs that are not $(1,1,2,5)$-colorable, not $(1,2,2,4)$-colorable, and not $(1,1,3,4)$-colorable.}

\begin{figure}
\begin{minipage}[b]{0.5\textwidth}
\hspace{5mm}
\begin{tikzpicture}[scale = 0.4, transform shape]

\node [ shape=circle, minimum size=0.1cm,  fill = black!1000, align=center] (v3) at (-12,6) {};

\node [ shape=circle, minimum size=0.1cm,  fill = black!1000, align=center]  (v1) at (-13,4) {};
\node [ shape=circle, minimum size=0.1cm,  fill = black!1000, align=center]  (v2) at (-11,4) {};
\draw  (v3) edge (v1);
\draw  (v1) edge (v2);
\draw  (v2) edge (v3);

\node [ shape=circle, minimum size=0.1cm,  fill = black!1000, align=center]  (v7) at (-9.5,3) {};
\node [ shape=circle, minimum size=0.1cm,  fill = black!1000, align=center]  (v9) at (-8.5,1) {};
\node  [ shape=circle, minimum size=0.1cm,  fill = black!1000, align=center] (v8) at (-7.5,3) {};
\draw  (v2) edge (v7);
\draw  (v7) edge (v8);
\draw  (v8) edge (v9);
\draw  (v9) edge (v7);
\node  [ shape=circle, minimum size=0.1cm,  fill = black!1000, align=center] (v11) at (-14.5,3) {};
\node  [ shape=circle, minimum size=0.1cm,  fill = black!1000, align=center] (v10) at (-16.5,3) {};
\node  [ shape=circle, minimum size=0.1cm,  fill = black!1000, align=center] (v12) at (-15.5,1) {};
\draw  (v10) edge (v11);
\draw  (v11) edge (v12);
\draw  (v12) edge (v10);
\draw  (v11) edge (v1);
\node  [ shape=circle, minimum size=0.1cm,  fill = black!1000, align=center] (v13) at (-5.5,3) {};
\node  [ shape=circle, minimum size=0.1cm,  fill = black!1000, align=center] (v14) at (-4,4.5) {};
\node  [ shape=circle, minimum size=0.1cm,  fill = black!1000, align=center] (v15) at (-3.5,2.5) {};
\draw  (v8) edge (v13);
\draw  (v13) edge (v14);
\draw  (v14) edge (v15);
\draw  (v15) edge (v13);

\node [ shape=circle, minimum size=0.1cm,  fill = black!1000, align=center] (v20) at (-8.5,-0.5) {};
\node [ shape=circle, minimum size=0.1cm,  fill = black!1000, align=center] (v21) at (-9.5,-2.5) {};
\node [ shape=circle, minimum size=0.1cm,  fill = black!1000, align=center] (v22) at (-7.5,-2.5) {};
\node [ shape=circle, minimum size=0.1cm,  fill = black!1000, align=center] (v17) at (-15.5,-0.5) {};
\node [ shape=circle, minimum size=0.1cm,  fill = black!1000, align=center] (v18) at (-16.5,-2.5) {};
\node [ shape=circle, minimum size=0.1cm,  fill = black!1000, align=center] (v19) at (-14.5,-2.5) {};
\node [ shape=circle, minimum size=0.1cm,  fill = black!1000, align=center](v4) at (-12,7.5) {};
\draw  (v3) edge (v4);
\node [ shape=circle, minimum size=0.1cm,  fill = black!1000, align=center] (v5) at (-18.5,3) {};
\node [ shape=circle, minimum size=0.1cm,  fill = black!1000, align=center] (v16) at (-20.5,2.5) {};
\node [ shape=circle, minimum size=0.1cm,  fill = black!1000, align=center] (v6) at (-20,4.5) {};
\draw  (v10) edge (v5);
\draw  (v5) edge (v6);
\draw  (v6) edge (v16);
\draw  (v16) edge (v5);
\draw  (v12) edge (v17);
\draw  (v17) edge (v18);
\draw  (v18) edge (v19);
\draw  (v19) edge (v17);
\draw  (v9) edge (v20);
\draw  (v20) edge (v21);
\draw  (v21) edge (v22);
\draw  (v22) edge (v20);
\node at (-3,2) {\huge{$u_9$}};
\node at (-4.5,5) {\huge{$u_8$}};
\node at (-7.7,1) {\huge{$u_6$}};
\node at (-5.5,3.75) {\huge{$u_7$}};

\node at (-11.1,6) {\huge{$u_1$}};
\node at (-7,2.5) {\huge{$u_5$}};
\node at (-10,2.5) {\huge{$u_4$}};
\node at (-13.5,4.5) {\huge{$u_2$}};
\node at (-10.5,4.5) {\huge{$u_3$}};
\draw  (-21.5,6.5) rectangle (-2,-3.5);
\node at (-11.5,8) {\huge{$v_3$}};
\node at (-11.65,-0.5) {\Huge{$G_3$}};
\end{tikzpicture}

\caption{Gadget $G_3$.}\label{g3}
\end{minipage}
\hspace{15mm}
\begin{minipage}[b]{0.4\textwidth}
\begin{tikzpicture}[scale = 0.4, transform shape]
\node [ shape=circle, minimum size=0.1cm,  fill = black!1000, align=center] (v3) at (-12,6) {};

\node [ shape=circle, minimum size=0.1cm,  fill = black!1000, align=center]  (v1) at (-13,4) {};
\node [ shape=circle, minimum size=0.1cm,  fill = black!1000, align=center]  (v2) at (-11,4) {};
\draw  (v3) edge (v1);
\draw  (v1) edge (v2);
\draw  (v2) edge (v3);

\node at (-11.3,6) {\huge{$v_1$}};

\node at (-13.5,4.5) {\huge{$v_2$}};
\node at (-10.5,4.5) {\huge{$v_3$}};

\node (v4) at (-12,7.5) {};
\node (v5) at (-9.5,3) {};
\node (v6) at (-14.5,3) {};
\draw  (v3) edge (v4);
\draw  (v2) edge (v5);
\draw  (v1) edge (v6);
\draw  (-14.5,10) rectangle (-9.5,7.5);
\draw  (-9.5,3) rectangle (-5,0.5);
\draw  (-19,3) rectangle (-14.5,0.5);
\node at (-12,8.7) {\huge{$G_3$}};

\node at (-7.15,1.75) {\huge{$G_3$}};
\node at (-16.7,1.7) {\huge{$G_3$}};

\end{tikzpicture}
\hspace{-5mm}
\caption{The construction $H$.}\label{example-2}
\end{minipage}

\end{figure}
%\end{figure}

We {first} show that there is a subcubic outerplanar graph that is not $(1,1,2,5)$-colorable.

\begin{example}
Our construction is the graph $G$ in Figure~\ref{g}, where each of the gadgets $G_2$ is the graph in Figure~\ref{g2} without the vertex $x_1$ (the graph surrounded by the rectangle), where each of the gadgets $G_1$ used in $G_2$ is the graph in Figure~\ref{g1} without the vertex $z_6$ (the graph surrounded by the rectangle). We show that $G$ is not packing $(1,1,2,5)$-colorable.
\end{example}

\begin{claim}\label{5g1}
In any packing $(1,1,2,5)$-coloring of $G_1$ vertex  $z_6$ cannot be colored with $5$.
\end{claim}

\begin{proof}
Suppose $z_6$ is colored with $5$. Then no vertex in $G_1$ can be colored with $5$ {since the farthest vertices from $z_6$ are $u_6, v_6$, each of them is at distance $5$ from $z_6$.} Then exactly one vertex of each of the four triangles, $u_1u_2u_3$, $u_4u_5u_6$, $v_1v_2v_3$, $v_4v_5v_6$ is colored with $2$. But the only way to use color $2$ in $u_1u_2u_3$ and $u_4u_5u_6$ is to color vertices $u_1$ and $u_6$ with $2$, and the only way to use color $2$ in $v_1v_2v_3$ and $v_4v_5v_6$ is to color vertices $v_1$ and $v_6$ with $2$, which is impossible since $u_1$ and $v_1$ are at distance two.
\end{proof}

Suppose $G$ has a packing $(1,1,2,5)$-coloring $f$. Then
\begin{equation}\label{25}
\mbox{each of triangles in $G$ has a vertex of color $2$ or $5$.}
\end{equation}
In particular, a vertex in $\{x_1,x_2,x_3\}$ is colored with $2$ or $5$.
By symmetry, we may assume $f(x_1)\in \{2,5\}$. By Claim~\ref{5g1} applied to the top of Fig. 1, we then have  $f(x_1)=2$. 
Since triangles $y_1y_2y_3$ and $t_1t_2t_3$ are too close to each other to both have a vertex of color $5$, in view of~\eqref{25}
 one of them has a vertex of  color $2$. By symmetry, we may assume it is $y_1y_2y_3$.
  Since $y_1$ is at distance two from $x_1$, one of $y_2$ and $y_3$, say $y_2$, is colored with $2$. Then $\{y_4, y_5, y_6\}$ does not have vertices of
  color $2$, and hence  {it has} a vertex of color
  $5$. By Claim~\ref{5g1} applied to right part of Fig. 1, this vertex is not $y_6$ and thus belongs to $\{y_4,y_5\}$. Then both  triangles $s_1s_2s_3$ and $z_1z_2z_3$ have to use color $2$, and since we cannot use $2$ at $z_1$ and $s_1$ at the same time, we may assume by symmetry that $z_2$ is colored with $2$. This implies we need to use $5$ at a vertex of the triangle $z_4z_5z_6$ and this vertex must be $z_6$ since $z_4$ and $z_5$ are at distance $5$ from $y_4$, which {contradicts} Claim~\ref{5g1}.

\vspace{4mm} 

 {Now we show there is a graph that is not $(1,2,2,4)$-colorable and not $(1,1,3,4)$-colorable.}

\begin{example}
 {Our construction is the graph $H$ in Figure~\ref{example-2}, where each of the gadgets $G_3$ is the graph in Figure~\ref{g3} without the vertex $v_3$ (the graph surrounded by the rectangle). We now show that $H$ is not $(1,2,2,4)$-colorable and not $(1,1,3,4)$-colorable.}
\end{example}

\begin{claim}
 {$H$ is not $(1,2,2,4)$-colorable.}
\end{claim}
\begin{proof}
 {Let the colors be $1, 2_a, 2_b, 4$. }

 {\textbf{Case 1:} The three colors used on $v_1, v_2, v_3$ are $1, 2_a, 2_b$. Say $v_3$ is colored with $2_b$ and $v_2$ is colored with $2_a$. Then the color $4$ must be used on the triangle $u_1u_2u_3$ and this vertex cannot be $u_1$ since otherwise we cannot color the triangle in the gadget hang on $v_2$ which corresponds to $u_1u_2u_3$ (we need to use color $4$ on this triangle as well). We assume by symmetry that $u_3$ is colored with $4$. Then one of $u_5$ and $u_6$ is colored with a color in $\{2_a, 2_b\}$, say $u_5$ is colored with $2_a$. But then we cannot use color $2_a$ and color $4$ on the triangle $u_7u_8u_9$, a contradiction.}

 {\textbf{Case 2:} The three colors used on $v_1, v_2, v_3$ are $1, 2_a, 4$. Say $v_2$ is colored with $4$ and $v_3$ is colored with $2_a$. Then the vertices $u_1, u_2, u_3$ have to choose colors from $\{1, 2_b\}$, a contradiction.} 

%Then the three colors used on $u_1, u_2, u_3$ have to be $1, 2_a, 2_b$. Then one of $u_2$ and $u_3$ is colored with a color $x \in \{2_a, 2_b\}$, say $u_3$ is colored with $x$. Then the vertices $u_4, u_5, u_6$ have to choose colors from $\{1, 2_a, 2_b\} - x$, a contradiction.

 {\textbf{Case 3:} The three colors used on $v_1, v_2, v_3$ are $2_a, 2_b, 4$. Say $v_2$ is colored with $4$ and $v_3$ is colored with $2_a$. Similarly to Case 2, we reach a contradiction.}
\end{proof}

\begin{claim}
 {$H$ is not $(1,1,3,4)$-colorable.}
\end{claim}
\begin{proof}
 {Let the colors be $1_a, 1_b, 3,4$. Since $v_1v_2v_3$ is a triangle, at least one of the vertices $v_1, v_2, v_3$, say $v_3$, is colored with a color in $\{3,4\}$.}

 {\textbf{Case 1:} $v_3$ is colored with $3$. Then $u_1$ cannot be colored with $4$ since otherwise the  triangle in the gadget hang on $v_2$ which corresponds to $u_1u_2u_3$  can only choose colors from $\{1_a, 1_b\}$, a contradiction. Thus one of the vertices in $\{u_2, u_3\}$ is colored with $4$, say $u_3$. Since $v_3$ and $u_4$ have distance $3$, $u_4$ cannot be colored with $3$ and one of $u_5, u_6$ is colored with $3$, say $u_5$. But then the vertices $u_7, u_8, u_9$ have to choose colored from $\{1_a, 1_b\}$, a contradiction.}

{\textbf{Case 2:} $v_3$ is colored with $4$. Then one of the vertices in $\{u_4, u_5, u_6\}$ is colored with $3$. But then the vertices $u_7, u_8, u_9$ have to choose colors from $\{1_a, 1_b\}$, a contradiction.}

\end{proof}

{\textbf{Acknowledgement. We thank the referees for their valuable comments.}}

%%%%%%%%%%%%%%%%%%%%%%%%%%%%%%%%%%%%%%%%%%%

\end{document}